\documentclass[fleqn,reqno,11pt]{article}

\usepackage[hmargin=1.25in, vmargin=1.25in, marginparwidth=1in, marginparsep=0.1in, a4paper, centering]{geometry}

\usepackage{amsmath, amsthm, amssymb, amsfonts}
\usepackage{mathtools} 
\usepackage[english]{babel}
\usepackage{cite}
\usepackage{fourier}
\usepackage[lining]{ebgaramond}
\usepackage{nicefrac}
\usepackage{enumerate}
\usepackage{pifont}

\usepackage{booktabs} 
\usepackage{texlogos} 

\usepackage{amsfonts}

\usepackage{upgreek}

\usepackage{xcolor}
\usepackage{graphicx}
\usepackage{comment} 
\usepackage{paralist} 
\usepackage{tikz}  
\usetikzlibrary{decorations.pathreplacing}

\usepackage[hang,flushmargin]{footmisc}

\usepackage[colorlinks=true,urlcolor=blue, citecolor=blue,linkcolor=blue,linktocpage,pdfpagelabels, bookmarksnumbered,bookmarksopen]{hyperref}
\usepackage[hyperpageref]{backref} 

\numberwithin{equation}{section}

\newtheorem{theorem}{Theorem}
\newtheorem*{theorem1}{Theorem}
\newtheorem{lemma}{Lemma}
\newtheorem{cor}{Corollary}
\newtheorem{prop}{Proposition}

\theoremstyle{definition}
\newtheorem{remark}{Remark}
\theoremstyle{remark}

\newcommand{\ds}{\displaystyle}

\newcommand{\R}{\mathbb{R}}

\newcommand{\de}{\partial}
\newcommand{\eps}{\varepsilon}

\def\XXint#1#2#3{{\setbox0=\hbox{$#1{#2#3}{\int}$}
   \vcenter{\hbox{$#2#3$}}\kern-.5\wd0}}

\renewcommand\footnotemark{}

\DeclareMathOperator{\dive}{div}

\title{Maximization of the efficiency of the first Dirichlet eigenfunction and improved eigenvalue inequalities}

\author{F. Della Pietra%
\thanks{\hspace*{-1.8em}Dipartimento di Matematica e Applicazioni ``R. Caccioppoli'', Universit\`a degli studi di Napoli Federico II, Via Cintia, Complesso Universitario Monte S. Angelo, 80126 Napoli, Italy. \href{mailto:f.dellapietra@unina.it}{\nolinkurl{f.dellapietra@unina.it}}}
}

\begin{document}
\maketitle

\begin{abstract}
We study the efficiency of the first Dirichlet eigenfunction $u$ on bounded convex domains $\Omega \subset \mathbb{R}^N$, defined as the ratio between the mean value of $u$ on $\Omega$ and its maximum value. By exploiting improved log-concavity estimates, we establish new sharp lower bounds for the first eigenvalue $\lambda_1$ and upper bounds for the efficiency in terms of the geometry of the domain, refining classical inequalities by Payne, Stakgold, and Hersch. Furthermore, we investigate the asymptotic behavior of the efficiency for elongating planar convex domains, making use of 1D limit profiles and Schr{\"o}dinger operators with convex potentials. As a main consequence of our analysis, we prove that among all planar convex domains the Payne-Stakgold upper bound is not optimal, and that there exists a maximizer of the efficiency.


\end{abstract}

\section{Introduction}
\label{sec:introduction}
Let $\Omega \subset \mathbb{R}^N$ ($N \ge 2$) be an open, bounded, convex set. Let $u \in H^1_0(\Omega)$ be the positive first Dirichlet eigenfunction on $\Omega$, with $\max_{\Omega} u=1$, which satisfies:
\begin{equation*} \label{eq:dirichlet}
    \begin{cases}
        \Delta u + \lambda_1(\Omega) u = 0 & \text{in } \Omega, \\
        u = 0 & \text{on } \partial\Omega,
    \end{cases}
\end{equation*}
where the first eigenvalue $\lambda_1(\Omega)$ is characterized by the Rayleigh quotient:
\begin{equation*}
    \lambda_1(\Omega) = \min_{\varphi\in H_0^1(\Omega)} \frac{\displaystyle \int_\Omega |\nabla \varphi|^2\,dx}{\displaystyle \int_\Omega \varphi^2\, dx}.
\end{equation*}

A cornerstone of shape optimization in geometric spectral theory is to establish sharp a priori estimates connecting $\lambda_1(\Omega)$ with the geometry of $\Omega$. One of the most celebrated results in this direction is the Faber-Krahn inequality, which states that among all domains of a given Lebesgue measure, the ball minimizes the first Dirichlet eigenvalue. However, when restricting the analysis to the class of convex domains, other  geometric parameters such as the inradius $R_{\Omega}$, the diameter $D_{\Omega}$ or the perimeter come into play.

A seminal result in this direction was obtained by Hersch \cite{hersch1960}, who found a sharp lower bound in terms of the inradius of a planar convex set:
\begin{equation} \label{eq:hersch}
    \lambda_1(\Omega) \ge \frac{\pi^2}{4}\frac{1}{R_\Omega^2}.
\end{equation}
Its validity in any dimension was proven by Protter \cite{protter_1981}.
The inequality is asymptotically attained as $\Omega$ tends to an infinite slab.

Another well-known lower bound for $\lambda_{1}(\Omega)$ under convexity assumptions was proven by Payne \cite{payne1980}: 
\begin{equation} \label{eq:payne_torsion}
    \lambda_1(\Omega) M_\Omega \ge \frac{\pi^2}{8}
\end{equation}
where $M_\Omega$ is the maximum of the torsion function $v_{\Omega}\in H_{0}^{1}(\Omega)$ such that $-\Delta v_{\Omega}=1$. Again, \eqref{eq:payne_torsion} is sharp and it is asymptotically attained as $\Omega$ degenerates into an infinite slab.

A third, now-classical inequality was found by Payne and Stakgold \cite{payne_stakgold1973} in 1973, relating to the average-to-peak ratio of the fundamental mode. Precisely, they showed that, if $u$ is the first positive eigenfunction of $-\Delta$ with $\max_{\Omega} u=1$ and $\Omega$ convex, then
\begin{equation}
\label{effintro}
E(\Omega) := \frac{1}{|\Omega|} \int_\Omega u \,dx \le \frac{2}{\pi};
\end{equation}
 the mean-to-peak ratio $E(\Omega)$ is also called the {\it efficiency}  of $\Omega$.
   It follows from the translation and scaling properties of the volume measure and eigenfunctions that 
\[
E(x_{0}+\alpha \Omega)=E(\Omega)
\]
for any $x_{0}\in \R^{N}$ and any $\alpha>0$.

Here, the optimality of \eqref{effintro} is a subtle question. For $N=1$, $E(\Omega)=2/\pi$. In the planar case, for any rectangle $R$, the efficiency is
\[
E(R)=\frac{4}{\pi^{2}}.
\]
On the other hand, as affirmed in \cite{payne_stakgold1973}, the constant $2/\pi$ is attained in two dimensions as the limit of a sequence of thinning annuli (see \cite{vdBDPDBG} for a proof); furthermore in \cite{vdBBK} it is shown that the trivial bound is reached: $\sup E(\Omega)=1$ among all connected domains of $\mathbb R^{N}$.
Finally, also assuming convexity, $E(\Omega)$ is not bounded from below by any positive constant. In two dimensions, examples of vanishing efficiency are given by sequences of elongating ellipses, triangles, rhombi or circular sectors; more generally, for domains where the corresponding eigenfunctions ``localise'' \cite{vdBDPDBG} (see also \cite{vdBBK,vdbB} and the references therein).

Several proofs of \eqref{eq:hersch}--\eqref{effintro} are available, making use of different techniques. One of these is based on the so called $P$-function method, introduced by Payne: a suitable function $P=P(u,\nabla u)$ satisfies an elliptic inequality and then the maximum principle; this allows to get a precise estimate of $\nabla u$. We refer to the classical monograph \cite{sperb} for an extensive overview of this method.

Refinements of \eqref{eq:hersch}--\eqref{effintro} are known in several directions. One natural question concerns the optimal class of domains for which such inequalities are valid: although they are known to hold for mean convex domains (where the mean curvature of the boundary is non-negative, see e.g. \cite{dpdbg}), it is of great interest to understand whether these estimates can be extended to broader geometric classes, at the cost of different sharp constants. For example, as regards \eqref{eq:hersch}, history and recent results on this subject can be found in \cite{bozzola_brasco}.

A parallel effort has been devoted to quantitatively sharpening the estimates themselves. To better capture the elongation and the asymmetry of the domain -- features that the inradius $R$ alone cannot fully capture on its own -- researchers have successfully enhanced these classical inequalities by including further geometric invariants, such as the diameter $D_\Omega$ or boundary curvature terms.

Regarding the Hersch-Protter inequality, Protter himself \cite{protter_1981} showed that
\begin{equation} \label{eq:hersch_protter2}
    \lambda_1(\Omega) \ge \frac{\pi^2}{4}\left(\frac{1}{R_\Omega^2}+\frac{N-1}{D_\Omega^{2}}\right).
\end{equation}
In \cite{mendez} it was proved that for convex planar domains
\begin{equation} \label{eq:hersch_hernandez}
    \lambda_1(\Omega) \ge \frac{\pi^2}{4}
    \left(\frac{1}{R_\Omega^2}+\frac{1}{(D_\Omega - R_\Omega)^{2}}\right),
\end{equation}
sharpening \eqref{eq:hersch_protter2} when $N=2$.

In the context of curvature-dependent refinements, some notable results were obtained by Payne and Philippin \cite{payne_philippin}. Using the log-concavity properties of the eigenfunctions \cite{acker}, they sharpened the classical gradient estimates into tighter bounds depending on $q_{min} = \min_{\partial \Omega} |\nabla u|$, thereby obtaining
\begin{equation*} \label{eq:payne_philippin_lambda}
    \lambda_1(\Omega) \ge \frac{\pi^2}{4}\frac{1}{R_\Omega^2} +\frac{q_{\text{min}}^2}{4}.
\end{equation*}
A similar technique was applied to the torsion function, yielding refined lower bounds for \eqref{eq:payne_torsion}, and could be extended to \eqref{effintro}. 

However, as explicitly noted by the authors themselves, making these curvature-based estimates fully explicit and applicable requires a strictly positive lower bound for $q_{\text{min}}$. 
Such lower boundary gradient estimates are difficult to obtain, especially for domains with high eccentricity where the gradient can become arbitrarily small. This intrinsic difficulty further motivates the search for sharp estimates that capture the domain's elongation relying exclusively on computable geometric invariants such as the diameter $D_\Omega$.

\medskip

The purpose of the present paper is then twofold: 
\begin{itemize}
\item to refine inequalities \eqref{eq:hersch}--\eqref{effintro} in the context of convex sets, introducing diameter-based remainder terms;
\item to show that in the planar convex case the efficiency inequality \eqref{effintro} is not optimal, clarifying the supremum of $E(\Omega)$ among planar convex domains.
\end{itemize}
 Specifically, our main results read as follows.

\begin{theorem}
\label{thm:herschbest}
Let $\Omega\subset \R^{N}$ be an open, bounded, convex set. Then
\begin{equation}
\label{herschbest}
    \lambda_1(\Omega) \ge \frac{\pi^2}{4}\left(\frac{1}{R_\Omega^2}+\frac{4(N-1)}{D_\Omega^{2}}\right).
\end{equation}
\end{theorem}
Now let be
  \begin{equation*} \label{eq:gamma_def}
        \delta_\Omega =\frac{(N-1)\pi^2}{\lambda_{1}(\Omega)D_\Omega^{2}-(N-1)\pi^2}>0.
    \end{equation*}
\begin{theorem} \label{thm:paynebest}
    Let $\Omega \subset \mathbb{R}^N$ be an open, bounded, convex  set.
    Then
    \begin{equation*} \label{eq:paynebest}
        \lambda_1(\Omega) M_\Omega \ge \frac{\pi^2}{8} + \delta_\Omega \sum_{k=1}^\infty \frac{2k}{(2k+1)^2(\delta_\Omega+1+2k)}.
    \end{equation*}
\end{theorem}

\begin{theorem}
\label{thm:paynestakgoldbest}
    Let $\Omega \subset \mathbb{R}^N$ be an open, bounded, convex set. Then:
    \begin{equation}
    \label{eq:eff1a}
E(\Omega)\le\frac{1}{\sqrt{\pi}} \frac{\Gamma\left(\frac{\delta_{\Omega}}{2}+1\right)}{\Gamma\left(\frac{\delta_{\Omega}+3}{2}\right)}=\frac{2}{\pi} \int_0^{\frac{\pi}{2}} (\cos x)^{\delta_{\Omega}+1} \, dx;
\end{equation}
in particular
   \begin{equation} \label{eq:eff1b}
                  E(\Omega) \le \frac{2}{\pi}\sqrt{1-\frac{\delta_{\Omega}}{2+\delta_{\Omega}}} .
    \end{equation}
\end{theorem}

\begin{remark}
In the planar case, since $D_{\Omega}\ge 2 R_{\Omega}$ trivially holds, inequality \eqref{herschbest} sharpens \eqref{eq:hersch_hernandez}. Moreover, the constant $\pi^2$ appearing in the remainder term ($N=2$), that is, on the right-hand side of
 \[
 		\lambda_1(\Omega) R_\Omega^2 - \frac{\pi^2}{4} \ge \pi^2 \frac{R_\Omega^2}{D_\Omega^2}
 \]
  is sharp. Indeed, considering a minimizing sequence of collapsing rectangles $\Omega_\varepsilon$ with sides $1$ and $\varepsilon$, 
 the first Dirichlet eigenvalue is $\lambda_1(\Omega_\varepsilon) = \pi^2(1 + 1/\varepsilon^2)$, while the inradius and diameter are given by $R_{\Omega_\varepsilon} = \varepsilon/2$ and $D_{\Omega_\varepsilon} = \sqrt{1+\varepsilon^2}$, respectively. Then
 \[
 \frac{\lambda_1(\Omega_\varepsilon) R_{\Omega_\varepsilon}^2 - \frac{\pi^2}{4}}{{R_{\Omega_{\eps}}^{2}/D_{\Omega_{\eps}}^{2}}}=\pi^{2}(1+\eps^{2}),
 \]  
 and we are done when $\eps \to 0^{+}$.
\end{remark}

\begin{remark}
\label{remeffnoeff}
Inequalities \eqref{eq:eff1a}--\eqref{eq:eff1b} are not effective for domains where $\delta_{\Omega}$ tends to 0, that is when $\lambda_{1} D_{\Omega}^{2}\to +\infty$, meaning that in this limit they do not provide any further information beyond the classical estimate \eqref{effintro}. Since
\[
j_{(N-2)/2,1}^{2} \left( \frac{D_{\Omega}}{R_{\Omega}} \right)^2\ge   \lambda_1(\Omega) D_{\Omega}^2 \ge \frac{\pi^2}{4} \left( \frac{D_{\Omega}}{R_{\Omega}} \right)^2,
\]
this happens if and only if
\[
\frac{D_\Omega}{R_\Omega} \to +\infty.
\]
To this matter, in Theorem \ref{thm:paynestakgoldmax} below we address the open problem regarding the optimality of $2/\pi$ for the efficiency when the dimension $N$ is fixed, giving a negative answer to this question in the planar case.
\end{remark}

\begin{theorem}
\label{thm:limeff}
    Let $\Omega_{n} \subset \mathbb{R}^2$ be a sequence of planar convex open sets with diameter $D_n$ and inradius $R_{n}$. Then if $D_{n}/R_{n} \to \infty$ as $n\to +\infty$, the efficiency of $\Omega_{n}$ is asymptotically strictly bounded by the two-dimensional rectangle:
    \begin{equation*}
        \limsup_{n \to \infty} E(\Omega_n) \le \frac{4}{\pi^2}.
    \end{equation*}
\end{theorem}
As a consequence, we establish the following result:
\begin{theorem}
\label{thm:paynestakgoldmax}
    The efficiency functional $E(\Omega)$ attains its supremum in the class of planar bounded convex open sets $\mathcal C$. That is, there exists $\Omega^{*}$ such that 
    \[
    E(\Omega^{*}) = \max \left\{E(\Omega),\;\Omega \in \mathcal C\right \}<\frac{2}{\pi}.
    \]
\end{theorem}

The paper is organized as follows. In Section 2, we provide the proofs of Theorems ~1--3. These are primarily based on a gradient estimate for the eigenfunctions, derived using the improved log-concavity estimate established in \cite{andrews_clutterbuck}. 

On the other hand, the proof of Theorem \ref{thm:paynestakgoldmax} is first based on characterizing the behavior of the efficiency for elongating convex sets (Theorem \ref{thm:limeff}), and this will be the subject of Section 3. A key point will be the asymptotic analysis carried out by Jerison and Grieser in \cite{jerison,grieser_jerison}. A careful analysis of the relevant Schr\"odinger one dimensional equation will be carried out (see Section 3.1). In particular, 
we show that the Payne-Stakgold estimate for the efficiency holds for the first eigenfunction of any Schr\"odinger one dimensional operators with convex potential (Proposition \ref{prop:integral_bound}). Finally, in sections 3.2 and 3.3 the complete proofs of Theorem \ref{thm:limeff} and Theorem \ref{thm:paynestakgoldmax}, respectively, will be given.  

\medskip

To conclude this introduction, we outline some open problems:
\begin{itemize}
\item finding the optimal explicit expressions for the remainder terms in the aforementioned inequalities;
\item determining the exact planar maximizer(s) $\Omega^*$ of the efficiency functional $E$;
\item extending Theorem \ref{thm:paynestakgoldmax} to arbitrary dimensions $N > 2$;
\item evaluating whether the $2/\pi$ efficiency bound holds for broader classes of domains in $\mathbb{R}^2$ (e.g., simply connected sets or, as some examples suggest, annular domains). If valid, is this constant optimal? Does this framework generalize to $N$ dimensions?
\end{itemize}

\section{Proofs of Theorems 1, 2 and 3}
From now on, it is not restrictive to assume that $\Omega \subset \mathbb{R}^N$ is a smooth, bounded, strictly convex open set. Let $u\in C^{2}(\overline\Omega)$ be the positive eigenfunction satisfying
\[
 -\Delta u=\lambda_{1}(\Omega) u,\qquad u=0\text{ on }\de\Omega,\qquad u>0 \text{ in }\Omega.
\]
The log-concavity of $u$ \cite{brascamp_lieb} guarantees the existence of a unique stationary point $x_{0}\in \Omega$, which coincides with its maximum. Hereafter, we assume
\[
u(x_{0})=1.
\]
Let us denote by $H_u=H_u(x)$ the mean curvature of the level set $\{u=t\}$, where $u(x)=t$:  
\[
H_u=-\textrm{div} \left(\frac{\nabla u}{|\nabla u|}\right).
\]
We recall the improved log-concavity estimate proven by Andrews and Clutterbuck \cite[Theorem 1.5]{andrews_clutterbuck} for positive Dirichlet eigenfunctions:
\begin{theorem1}[Improved log-concavity]
If $w=\log (u)$, then for every $\eta \in \mathbb{R}^N$ and for every $x \in \Omega$ we have
\[
-\left\langle \nabla^2 w(x) \eta, \eta\right\rangle \ge \frac{\pi^2}{D_{\Omega}^2}|\eta|^2.
\]
\end{theorem1}
As observed in \cite{cristoforoni_nitsch_trombetti}, the above result implies that for all $x \in \Omega$ such that $\nabla u(x) \neq 0$
\begin{equation} \label{eq:AC_bound}
|\nabla u(x)| H_u(x) \ge \frac{(N-1)\pi^2}{D_{\Omega}^2} u(x).
\end{equation}
\begin{remark}
By taking the trace of the Hessian of $w=\log u$ and using the improved log-concavity, it follows that $-\Delta w = \lambda_1 + \frac{|\nabla u|^2}{u^2} \ge \frac{N\pi^2}{D_\Omega^2}$ in $\Omega$; evaluating this inequality at the maximum point yields
\begin{equation*}
    \label{eq:lambda_bound}
    \lambda_1 (\Omega) \ge \frac{N\pi^2}{D_\Omega^2}. 
\end{equation*}
\end{remark}
\begin{prop} \label{thm:gradient_bound}
The eigenfunction $u$ satisfies
\begin{equation} \label{eq:grad_est}
|\nabla u|^2 \le \tilde{\lambda} (1-u^2) \quad \text{in } \Omega,
\end{equation}
where $\tilde{\lambda} = \lambda_1 - \frac{(N-1)\pi^2}{D_{\Omega}^2}$.
\end{prop}
Inequality \eqref{eq:grad_est} is a refinement of a well-known gradient estimate for the first eigenfunction, stated replacing $\tilde\lambda$ with $\lambda_{1}$ (see \cite{payne_stakgold1973,sperb}).

\begin{proof}
Fix $\eps > 0$. By standard regularity results, $u \in C^2(\overline\Omega)$. Let us consider $P_\eps \in C^1(\overline{\Omega})$ defined by:
\begin{equation*} \label{eq:Phi_eps_def}
P_\eps(x) = \frac{|\nabla u|^2}{2} + (\tilde{\lambda}+\eps)\frac{u^2}{2} ,
\end{equation*}
with $\eps>0$. The function $P_\eps$ attains its maximum at some point $x_1 \in \overline{\Omega}$. We claim that $x_{1}=x_{0}$, the maximum point of $u$. If this holds, then for any $\eps>0$ we have
\[
|\nabla u(x)|^2 + (\tilde{\lambda}+\eps) u(x)^2 \le \tilde{\lambda}+\eps
\]
that implies the thesis.

To prove the claim, we first show that $x_1$ cannot lie on the boundary of $\Omega$. Assume by contradiction that $x_1 \in \partial \Omega$; then 
\begin{equation}
    \label{eq:Phi_eps_boundary}
\frac{\partial P_\eps}{\partial \nu}(x_1) \ge 0.
\end{equation}
Since $u=0$ on $\partial \Omega$, it holds that
\[
\frac{\partial P_\eps}{\partial \nu} =   u_\nu u_{\nu\nu} \qquad\text{on }\partial \Omega.
\]
On the other hand, along any regular level set $\{u=t\}$, the Laplacian can be expressed as
\[
\Delta u = u_{\nu\nu} + H_{u} u_\nu, 
\]
and using the equation and the boundary condition, we find $u_{\nu\nu} = -H_u u_\nu$ on $\partial \Omega$. Thus:
\begin{equation*}
\frac{\partial P_\eps}{\partial \nu}(x_1) = -H_{u}(x_1) (u_\nu(x_1))^2.
\end{equation*}
The strict convexity of the boundary implies $H_u > 0$, while Hopf's Lemma gives $u_\nu(x_1) < 0$. Together, they yield $\frac{\partial P_\eps}{\partial \nu}(x_1) < 0$, which contradicts \eqref{eq:Phi_eps_boundary}. Hence,  $x_{1}\in\Omega$.

The next step consists of proving that $x_1$ is a critical point of $u$. To this end, let us observe that at $x_1$ we have $\nabla P_\eps = \nabla^2u \nabla u + (\tilde{\lambda}+\eps )u \nabla u = 0$, hence
\begin{equation} \label{eq:first_order_eps}
 \nabla^2u \nabla u = -(\tilde{\lambda}+\eps) u \nabla u\quad\text{at }x_{1}. 
\end{equation}
Assume by contradiction that $\nabla u(x_1) \ne 0$. Then at $x_1$ we have
\begin{equation}
\label{eq:Hu_fi}
H_u=-\dive \left( \frac{\nabla u}{|\nabla u|} \right)  = \frac{-\Delta u}{|\nabla u|} + \frac{\langle \nabla^2u \nabla u, \nabla u \rangle}{|\nabla u|^3}.
\end{equation}
Hence, substituting the eigenvalue equation and \eqref{eq:first_order_eps} into \eqref{eq:Hu_fi} yields, at $x_1$,
\begin{align*} \label{eq:Hu_final_derivation}
|\nabla u|H_u & = {\lambda_1 u} -{\tilde{\lambda}_\eps u} \nonumber \\
&= \left(\frac{(N-1)\pi^2 }{D^2}-\eps\right)u,
\end{align*}
contradicting \eqref{eq:AC_bound}. Therefore, $\nabla u(x_1) = 0$, and the claim is proved.
\end{proof}
The aforementioned gradient estimate is the starting point for proving our first three main results.
\begin{proof}[Proof of Theorem \ref{thm:herschbest}]
 Let $r$ denote the arc length parameter along the radial segment from $x_0$ to $x\in\partial\Omega$. Then, by \eqref{eq:grad_est}, we have
    \begin{equation*}
   -\frac{\partial u}{\partial r_x} \le  \frac{|\nabla u|}{\sqrt{1-u^2}} \le \sqrt{\tilde{\lambda}}
    \end{equation*}
Integrating along the segment connecting $x_0$ and $x\in \partial\Omega$, we get 
 \[
\sqrt {\tilde{\lambda}} |x_0-x|  \ge \int_0^1 \frac{1}{\sqrt{1-u^2}} du =\frac{\pi}{2},
 \] 
 which gives \eqref{herschbest}.
\end{proof}    
In order to prove Theorem \ref{thm:paynebest} and Theorem \ref{thm:paynestakgoldbest}, we need the following
\begin{prop} \label{thm:improved_payne_bounds}
    Let $v_{\Omega}$ be the torsion function in $\Omega$, with $M_{\Omega}=\max_{\Omega} v_{\Omega}$, and consider 
     $\gamma :=\frac{\lambda_1}{\tilde\lambda}=1+\delta_\Omega.
    $  Defining $\Phi: [0, 1] \to [0,+\infty)$ by
    \begin{equation*} \label{eq:barrier_def}
        \Phi(s) = \gamma \int_s^1 (1-\tau^2)^{-\frac{\gamma}{2}} \left( \int_\tau^1 (1-r^2)^{\frac{\gamma}{2}-1} \, dr \right) d\tau,
    \end{equation*}
    with $\Phi_0 = \Phi(0)$, the following inequalities hold:
    \begin{equation} \label{eq:pointwise_ineq}
        \Phi_0 - \Phi\left(u\right) \le  \lambda_1 M_\Omega \quad \text{in } \Omega,
    \end{equation}
and
    \begin{equation} \label{eq:boundary_grad}
        |\nabla u| \le \frac{\lambda_1}{\sqrt{\pi}} \frac{\Gamma\left(\frac{\gamma+1}{2}\right)}{\Gamma\left(\frac{\gamma}{2}+1\right)} |\nabla v_\Omega| \quad \text{on } \partial \Omega.
    \end{equation}
\end{prop}

\begin{proof}
Let $s \in [0, 1]$. Define
    \begin{equation*}
        w(s) = -\Phi'(s) = \gamma (1-s^2)^{-\frac{\gamma}{2}} \int_s^1 (1-r^2)^{\frac{\gamma}{2}-1} \, dr = \frac{\gamma}{2} \int_0^1 \frac{t^{\frac{\gamma}{2}-1}}{\sqrt{1-(1-s^2)t}} \, dt.
    \end{equation*}
    Notice that $w(s) \ge 0$ for $s \in [0, 1)$. Since $\gamma = \lambda_1 / \tilde{\lambda}$, it follows that $w(s)$ satisfies
    \begin{equation} \label{eq:w_ode}
        w'(s)\tilde{\lambda}(1-s^2) - \lambda_1 s w(s) = -\lambda_1.
    \end{equation}
    Furthermore, it is easy to check that $w'(s) \le 0$ on $[0,1)$. Consequently, $\Phi''(s) = -w'(s) \ge 0$. Computing the Laplacian of $\Phi(u)$ and using the eigenvalue equation, we obtain
    \begin{equation*}
        \Delta \left[ \Phi\left(u\right) \right] = \Phi''(u) |\nabla u|^2 - \lambda_1 u \Phi'({u}).
    \end{equation*}
    Since $\Phi'' \ge 0$, by \eqref{eq:grad_est} it holds that
    \begin{align*}
        \Delta \left[ \Phi(u) \right] &\le \Phi''(u)\tilde\lambda(1-u^2) - \lambda_1 u \Phi'(u) \\
        &= -w'(u)\tilde\lambda(1-u^2) + \lambda_1 u w(u)=\lambda_{1},
    \end{align*}
    where the last equality follows from \eqref{eq:w_ode}.
    The auxiliary function $\zeta= \lambda_1 v_\Omega + \Phi(u) - \Phi_0$ then satisfies
    \begin{equation*}
        \Delta \zeta= \lambda_1 \Delta v_\Omega + \Delta \Phi(u) = -\lambda_1 + \Delta \Phi(u) \le 0,
    \end{equation*}
    thus $\zeta$ is superharmonic in $\Omega$. Since $\zeta= 0$ on $\partial \Omega$, the maximum principle yields $\zeta\ge 0$ throughout $\Omega$, proving \eqref{eq:pointwise_ineq}.
    
    To derive the boundary gradient inequality, we observe that since $\zeta\ge 0$ in $\Omega$ and $\zeta= 0$ on $\partial \Omega$, it holds that
    \begin{equation}
        \label{pass_1}
        \frac{\partial \zeta}{\partial \nu} = \lambda_1 \frac{\partial v_\Omega}{\partial \nu} + \Phi'(0) \frac{\partial u}{\partial \nu} \le 0,
    \end{equation}
    where $\nu$ is the outward normal. Since $u$ and $v_\Omega$ are positive in $\Omega$ and vanish on $\partial \Omega$, their outward normal derivatives satisfy $\frac{\partial u}{\partial \nu} = -|\nabla u|$ and $\frac{\partial v_\Omega}{\partial \nu} = -|\nabla v_\Omega|$. Substituting, and noting that $\Phi'(0) = -w(0)$, from \eqref{pass_1} we get
    \begin{equation*}
  |\nabla u| \le \frac{\lambda_1}{w(0)} |\nabla v_\Omega|.
    \end{equation*}
    The constant $w(0)$ can be explicitly computed as:
    \begin{equation*}
        w(0) = \gamma \int_0^1 (1-r^2)^{\frac{\gamma}{2}-1} \, dr = \frac{\gamma}{2} B\left(\frac{1}{2}, \frac{\gamma}{2}\right) = \frac{\gamma\Gamma\left(\frac{1}{2}\right)\Gamma\left(\frac{\gamma}{2}\right)}{2\Gamma\left(\frac{\gamma+1}{2}\right)} = \sqrt{\pi} \frac{\Gamma\left(\frac{\gamma}{2}+1\right)}{\Gamma\left(\frac{\gamma+1}{2}\right)},
    \end{equation*}
yielding \eqref{eq:boundary_grad}.
\end{proof}
    Evaluating inequality \eqref{eq:pointwise_ineq} at the point $x_0 \in \Omega$, where $u(x_0) = 1$, we deduce the following corollary:
\begin{cor}
Under the assumptions of Proposition \ref{thm:improved_payne_bounds}, it holds that
 \begin{equation}
 \label{thm:improved_payne_bounds_eq}
        \lambda_1 M_\Omega \ge \Phi_0.
    \end{equation}
\end{cor}
\begin{proof}[Proof of Theorem \ref{thm:paynebest}]
This result essentially follows from the prior proofs. Applying Fubini's theorem, the constant $\Phi_0 = \Phi(0)$ can be rewritten as
\begin{equation*}
    \Phi_0 = \gamma \int_0^1 (1-r^2)^{\frac{\gamma-2}{2}} \left( \int_0^r (1-\tau^2)^{-\frac{\gamma}{2}} \, d\tau \right) dr,
\end{equation*}
and expanding the inner integrand into a power series and integrating term-by-term, a straightforward computation yields
\begin{equation*}
    \Phi_0 = \frac{\pi^2}{8} + (\gamma-1) \sum_{k=1}^\infty \frac{2k}{(2k+1)^2(\gamma+2k)}.
\end{equation*}
Substituting this identity back into \eqref{thm:improved_payne_bounds_eq}, with $\delta_{\Omega}=\gamma-1$, the proof is concluded.
\end{proof}
   
By using Proposition \ref{thm:improved_payne_bounds} we get the proof of Theorem \ref{thm:paynestakgoldbest}.
%
 \begin{proof}[Proof of Theorem \ref{thm:paynestakgoldbest}]
Integrating the equations for the eigenvalue and the torsional rigidity over $\Omega$  yields 
\[
E(\Omega)= \frac{\ds\int_{\Omega} -\Delta u dx}{\ds\lambda_{1}(\Omega) \int_{\Omega}1\,dx}= \frac{\ds\int_{\partial \Omega} |\nabla u| \, d\mathcal H^{N-1}}{\lambda_{1}(\Omega)\ds\int_{\partial \Omega} |\nabla v_\Omega| \, d\mathcal H^{N-1}}.
\]
By applying \eqref{eq:boundary_grad}, we get
     \[
     E(\Omega) \le \frac{1}{\sqrt{\pi}} \frac{\Gamma\left(\frac{\gamma+1}{2}\right)}{\Gamma\left(\frac{\gamma}{2}+1\right)}
     \]
that is \eqref{eq:eff1a}. As regards \eqref{eq:eff1b}, by exploiting the log-convexity of the Gamma function, it holds that 
\begin{equation}
\label{gammamon}
x \mapsto \frac{\Gamma(x+1/2)\sqrt{x+1/2}}
{\Gamma(x+1)}
\quad\text{is decreasing for }x>0
\end{equation}
(See Lemma \ref{append} below). Hence
 \[
 \sup_{x\ge \frac{1}{2}} \frac{\Gamma(x+1/2)\sqrt{x+1/2}}{\Gamma(x+1)} =\frac{2}{\sqrt{\pi}}.
 \]
 For $x = \gamma/2 \ge 1/2$, we obtain 
    \begin{equation*}
        \frac{\Gamma\left(\frac{\gamma+1}{2}\right)}{\Gamma\left(\frac{\gamma}{2}+1\right)} \le \frac{2}{\sqrt{\pi}} \sqrt{\frac{2}{\gamma+1}}.
    \end{equation*}
 Then substituting in \eqref{eq:eff1a} and 
    recalling the definition of $\delta_{\Omega}+1=\gamma = \frac{\lambda_1}{\lambda_1 - \frac{\pi^{2}(N-1)}{D^{2}}}$, we get the thesis.
 \end{proof}

For the sake of completeness, we establish the monotonicity property \eqref{gammamon} in the following lemma. We briefly recall the well-known definitions of the special functions involved. Euler's Gamma function is defined for $z > 0$ by the 
\begin{equation*}
    \Gamma(z) = \int_0^\infty t^{z-1} e^{-t} \, dt.
\end{equation*}
Its logarithmic derivative is the digamma function, denoted by $\psi(z)$:
\begin{equation*}
    \psi(z) = \frac{d}{dz} \ln \Gamma(z) = \frac{\Gamma'(z)}{\Gamma(z)}.
\end{equation*}

\begin{lemma} \label{append}
    The auxiliary function
    \begin{equation*}
        f(x) = \frac{\Gamma\left(x+\frac{1}{2}\right)}{\Gamma(x+1)} \sqrt{x+\frac{1}{2}}
    \end{equation*}
    is strictly decreasing for all $x > -\frac{1}{2}$.
\end{lemma}

\begin{proof}
    We consider the logarithmic derivative of $f(x)$:
    \begin{equation*}
        \frac{d}{dx} \ln f(x) = \psi\left(x+\frac{1}{2}\right) - \psi(x+1) + \frac{1}{2x+1}.
    \end{equation*}
    To prove that $f'(x) < 0$, we show that
    \begin{equation*}
        \psi(x+1) - \psi\left(x+\frac{1}{2}\right) > \frac{1}{2x+1}.
    \end{equation*}
    Using the standard integral representation of the digamma function (see \cite[Eq. 6.3.21]{abramowitz1964}), the difference can be explicitly written as:
    \begin{equation*}
        \psi(x+1) - \psi\left(x+\frac{1}{2}\right) = \int_0^\infty \frac{e^{-\left(x+\frac{1}{2}\right)t} - e^{-(x+1)t}}{1-e^{-t}} \, dt=\int_0^\infty  \frac{e^{-\left(x+\frac{1}{2}\right)t}}{1+e^{-t/2}} \, dt.
    \end{equation*}
    Since 
        \begin{equation*}
        \int_0^\infty e^{-\left(x+\frac{1}{2}\right)t} \frac{1}{1+e^{-t/2}} \, dt > \frac{1}{2} \int_0^\infty e^{-\left(x+\frac{1}{2}\right)t} \, dt=\frac{1}{2x+1},
    \end{equation*}
we get $f'<0$ and the thesis follows.
\end{proof}

\section{Asymptotic behavior of $E(\Omega)$ and proofs of Theorems \ref{thm:limeff} and \ref{thm:paynestakgoldmax}}
As observed in Remark \ref{remeffnoeff}, inequalities \eqref{eq:eff1a}-\eqref{eq:eff1b} are not effective for domains where 
$\frac{D_\Omega}{R_\Omega} \to +\infty$.

Then to  analyze the limit of the efficiency of any thinning or elongating planar convex domains, we have to perform a different analysis. We rely on asymptotic theorems governing the behavior of the first Dirichlet eigenfunction, established by Jerison \cite{jerison}, Grieser and Jerison \cite{grieser_jerison}, and study the relevant one dimensional Schr\"odinger equation.

Throughout this Section we will assume that $\Omega$ is normalized, in the sense that:
\begin{itemize}
\item $\Omega$ is a bounded convex planar domain;
\item $\Omega$ is rotated such that its projection onto the $y$-axis has the minimum possible length, and dilated so that this minimal projection length equals $1$. 
\end{itemize}
Thus $\Omega$ can be written as 
\[
\Omega=\left\{(x, y): f_1(x)<y<f_2(x), a< x< b\right\}
\]
where $f_1$ and $f_2$ are convex and concave, respectively, on $(a,b)$ and satisfy
\[
 0 \leq f_1(x) \leq f_2(x) \leq 1 \quad \text { for } \quad a \leq x \leq b, \quad
 \quad \max _{[a, b]} f_2=1 , \quad \min _{[a, b]} f_1=0.
\]
Let $h(x)$ be the height of $\Omega$ at $x\in(a,b)$, that is
\[
h(x)=f_2(x)-f_1(x),
\]
with $h(x_{1})=\max_{[a,b]} h(x)=1$, and consider in $[a,b]$
\[
\mathcal{L}=\frac{d^2}{d x^2}-\frac{\pi^2}{h(x)^2}
\]
with homogeneous Dirichlet boundary conditions. 
Let $\mu_{1}$ be the first eigenvalue of $\mathcal{L}$ and let $\phi$ be the first positive eigenfunction:
\[
(\mathcal{L}+\mu_{1}) \phi=0 \text { on }(a, b), \quad \phi(a)=\phi(b)=0,
\]
normalized as $\max\phi=1$.

\begin{figure}[h]
    \centering
    \begin{tikzpicture}[scale=.8,xscale=1.1, yscale=3]
        
        \draw[->, thin] (-0.5, 0) -- (12.5, 0) node[right] {$x$};
        \draw[->, thin] (0, -0.1) -- (0, 1.2) node[right] {$y$};
        
        \node[below left] at (0,0) {0};
        \draw[dotted, thick] (0, 1) -- (2.5, 1);
        \draw (-0.1, 1) -- (0, 1) node[left] {1};
        
        \draw[thick] (0.5, 0.4) 
            .. controls (0.5, 1.0) and (1.5, 1.0) .. (3.0, 1.0) 
            .. controls (6.0, 1.0) and (9.0, 0.8) .. (11.5, 0.5);
            
        \draw[thick] (0.5, 0.4) 
            .. controls (0.5, 0.1) and (1.5, 0.0) .. (3.0, 0.0) 
            .. controls (6.0, 0.0) and (9.0, 0.2) .. (11.5, 0.5);

        \node at (5.5, 0.45) {\large $\Omega$};
        \node[above] at (8, 0.82) {$y = f_2(x)$};
        \node[below] at (8, 0.18) {$y = f_1(x)$};
        
        \draw[dotted, thick] (0.5, 0.4) -- (0.5, 0) node[below] {$a$};
        
        \draw[dotted, thick] (11.5, 0.5) -- (11.5, 0) node[below] {$b$};
        
        \coordinate (MaxU) at (2.5, 0.45);
        \draw[fill=white] (MaxU) circle (0.8pt); 
        \draw[dotted, thick] (MaxU) -- (2.5, 0) node[below] {$x_0$};
        \draw (2.5, -0.03) -- (2.5, 0.03); 
        \node[above right] at (MaxU) {\textsf{Max of} $u$};
        
        \draw[dotted, thick] (3.2, 0.01) -- (3.2, 0) node[below] {$x_1$};
       \draw (3.2, -0.03) -- (3.2, 0.03); 
        
        \coordinate (Xtop) at (9.5, 0.74);
        \coordinate (Xbot) at (9.5, 0.24);
        \draw[dotted, thick] (Xtop) -- (9.5, 0) node[below] {$x$};
        
        \draw[decorate,decoration={brace,amplitude=4pt}, thick] 
            (9.5, 0.3) -- (9.5, 0.7) node[midway, left=4pt] {$h(x)$};
    \end{tikzpicture}
\end{figure}

The total length of the domain along the $x$-axis is denoted by $N = b - a$.
If $\Omega$ is a rectangle, then $h$ is constant, $\lambda_{1}(\Omega)=\mu_{1}$ and
\[
u(x,y)= \phi(x) \sin \pi y =\sin \left(\pi\frac{x-a}{N}\right) \sin \pi y.
\]
The longitudinal behavior of the first eigenvalue is governed not merely by $N$, but rather by an effective length scale $L$. This parameter is defined as the length of the longest interval $I \subset [a, b]$ where the domain remains sufficiently close to its maximum height $1$:
\begin{equation*}
    h(x) \ge 1 - \frac{1}{L^2} \quad \text{for all } x \in I.
\end{equation*}
Due to the concavity of $h(x)$, the sharpest possible height drop from the maximum occurs when the domain is a triangle. This yields 
\begin{equation*} \label{eq:L_N_bounds}
    \sqrt[3]{N} \le L \le N.
\end{equation*}
Moreover, the convexity of $\Omega$ gives
\[
N\ge |\Omega| \ge \frac{N}{2},
\]
with equalities for rectangles and triangles respectively.

On the other hand, it is straightforward to show (see \cite{jerison}) that the first eigenvalue $\mu_{1}$ of $\mathcal L$ satisfies, for large $L$,
\begin{equation}
\label{bound1dim}
    \pi^2 < \mu_1 \le \pi^2\left(1 + \frac{3}{L^2}\right) +O\left(\frac{1}{L^{4}}\right)
\end{equation}

The correspondence between the two-dimensional eigenfunction $u(x,y)$ and the one-dimensional longitudinal profile $\phi_{1}(x)$ is formalized by the following result.
\begin{theorem1}[{Grieser-Jerison \cite[Theorem 1.6]{grieser_jerison}}]
     Let $I'$ be the interval concentric with $I$ of length $\frac N2$. There exists an absolute constant $C > 0$ such that
    \begin{equation*}
        \left| u(x,y) - \phi_{1}(x)\sin \alpha(x,y) \right| \le \frac{C}{L} \quad \text{for all } x \in I'
    \end{equation*}
    where
    \[
        \alpha(x,y) = \pi \frac{y-f_1(x)}{h(x)},\quad  y\in [f_{1}(x),f_{2}(x)].
    \]
\end{theorem1}

\subsection{Asymptotic Analysis of the $L^1$ Norm of the First One-Dimensional Eigenfunction}
The first step of our analysis consists in capturing the asymptotic behavior of the $L^1$ norm of the positive normalized first one-dimensional eigenfunction $\phi_{1}$, which solves
  \begin{equation*}
\begin{cases}
    -\phi_1''(x) + V(x)\phi_1(x) = \mu_1 \phi_1(x) \quad \text{in } (a, b),
   \\
   \phi_1(a) = \phi_1(b) = 0, \\
   \phi_1\ge0, \; \max \phi_1=1.
\end{cases}
\end{equation*}
where $V(x) = \pi^2 / h(x)^2$ represents a convex potential due to the strict concavity of $h(x)$. In the following result, we establish an asymptotic upper bound relative to the domain area. 

First of all, we show that the Payne-Stakgold estimate for the efficiency holds for the first eigenfunction of any Schr\"odinger one dimensional operators with convex potential.
\begin{prop}
\label{prop:integral_bound}
Let $(\alpha,\beta)$ be a bounded interval, and let $W(s)$ be a convex potential defined on $(\alpha,\beta)$. Consider a positive solution $v_1$ such that
  \begin{equation*}
\begin{cases}
    -v_{1}'' + W(s)v_{1}(s) = \mu_{1} v_1(s) \quad \text{in } (\alpha, \beta),
   \\
   v_1(\alpha) = v_1(\beta) = 0, \\
   v_{1}\ge 0, \; \max v_{1}=1,
\end{cases}
\end{equation*}
where $\mu_{1}$ is the first eigenvalue. Then
\[
\int_{\alpha}^{\beta}v_1(s)ds \le \frac{2}{\pi}(\beta-\alpha).
\]
\end{prop}
\begin{proof}
Let $w(s) = \log v_{1}(s)$. A key ingredient is the following modulus of log-concavity estimate (see \cite{andrews_clutterbuck}):
\begin{equation}
\label{ace}
\left( w'(y) - w'(x) \right) \frac{y-x}{|y-x|} \le -\frac{2\pi}{D} \tan\left( \frac{\pi}{2D} |y-x| \right),
\end{equation}
for any $x,y \in (\alpha,\beta)$, with $D=\beta-\alpha$.

For any $t \in (0, 1)$, define $J_t = \{s \in (\alpha, \beta) : v_1(s) > t\}$. Due to the log-concavity of $v_1$, $J_t$ is an interval $(L_t, R_t)$ such that $v_1(L_t) = v_1(R_t) = t$. The length of this level set is denoted by 
\[
\ell(t) = R_t - L_t.
\]
Let $p: = \log t < 0$, so that $w(L_t) = w(R_t) = p$. The midpoint of $J_{t}$ is $M = \frac{L_t + R_t}{2}$.
For $\xi \in \left[0, \frac{\ell(t)}{2}\right]$, define:
\[
y(\xi) = R_t - \xi, \quad
x(\xi) = L_t + \xi.
\]
Notice that $y(\xi) - x(\xi) = \ell(t) - 2\xi \ge 0$. Substituting these into \eqref{ace} yields 
\begin{equation}
\label{4}
w'(R_t - \xi) - w'(L_t + \xi) \le -\frac{2\pi}{D} \tan\left( \frac{\pi}{2D} (\ell(t) - 2\xi) \right).
\end{equation}
We integrate both sides of inequality \eqref{4} with respect to $\xi$ over the interval $\left[0, \frac{\ell(t)}{2}\right]$. For the left-hand side, we get
\[
\int_0^{\ell(t)/2} \left[ w'(R_t - \xi) - w'(L_t + \xi) \right] d\xi 
= 2p - 2w(M).
\]
For the right-hand side, a straightforward computation gives
\begin{equation*}
\int_0^{\ell(t)/2} -\frac{2\pi}{D} \tan\left( \frac{\pi}{2D} (\ell(t) - 2\xi) \right) d\xi 
= 2 \log \cos\left( \frac{\pi \ell(t)}{2D} \right).
\end{equation*}
Combining the two integrals, and recalling that $w(M)\le 0$, we find
\begin{equation*}
p \le \log \cos\left( \frac{\pi \ell(t)}{2D} \right).
\end{equation*}
Being $t= e^{p}$, this implies
\begin{equation*}
\ell(t) \le \frac{2D}{\pi} \arccos(t).
\end{equation*}
By Cavalieri principle, this yields
\begin{equation*}
\int_\alpha^\beta v_1(s)\,ds = \int_0^1 \ell(t)\,dt \le \frac{2D}{\pi} \int_0^1  \arccos(t)\,dt= \frac{2}{\pi} D,
\end{equation*}
and the proof is completed.
\end{proof}

Now, let us consider the first one dimensional eigenfunction $\phi_{1}$ of the relevant Schr\"odinger operator $\mathcal L$.
\begin{prop}
\label{prop:limit-uni}
Let $\Omega_{L} \subset \mathbb{R}^2$ be a family of normalized bounded planar convex open sets, where the parameter $L \to \infty$ denotes their effective length. It holds that
\begin{equation}
\label{limit-uni}
    \limsup_{L \to \infty} \frac{1}{|\Omega_L|} \int_a^b \phi_1(x) \,dx \le \frac{2}{\pi}.
\end{equation}
\end{prop}

\begin{proof}
First, let us observe that
\begin{equation*}
    |\Omega_L| = \int_a^b h(x) \,dx \ge \int_I h(x) \,dx.
\end{equation*}
By the definition of $L$, we have $h(x) \ge 1 - 1/L^2$ on $I$. Thus:
\begin{equation*} \label{eq:area_bound}
 \frac{N}{L} \ge  \frac{|\Omega_L|}{L} \ge   1 - \frac{1}{L^2}.
\end{equation*}
 We now consider the two possible asymptotic regimes for the sequence of domains.

\textbf{Case 1: $N/L \to +\infty$.} This is what happens, for example, in elongating ellipses (where $L \sim \sqrt{N}$), or triangular domains (where $ L\sim \sqrt[3]{N}$). In \cite{grieser_jerison} it is shown that
\begin{equation*}
    \int_a^b \phi_1(x) \,dx \le C L,
\end{equation*}
for some positive constant $C$. Therefore \eqref{limit-uni} trivially holds. Indeed, since $|\Omega| \ge N/2$, we obtain $\frac{1}{|\Omega|} \int_a^b \phi_1(x) \,dx \le \frac{2 C L}{N} \to 0$.

\textbf{Case 2: $N/L \to \gamma\in [1,+\infty[$}.  By translation invariance, we can assume without loss of generality that $I = [-L/2, L/2]$. Consider the rescaled eigenfunction $v_L(y) = \phi_1(Ly)$, $y\in[A_{L},B_{L}]=[a/L,b/L]$. In this flat regime, we have $B_L - A_L = N/L\to \gamma$ as $L \to \infty$. Being $I\subseteq [a, b]$, it holds that $A_L \le -\frac{1}{2}$ and $B_L \ge \frac{1}{2}$. Hence, up to extracting a subsequence, their limits $A_{L}\to A_\infty$ and $B_{L}\to B_\infty$ satisfy
\begin{equation*} \label{eq:finite_domain}
\frac{1}{2} - \gamma \le A_\infty \le -\frac{1}{2},\qquad
\frac{1}{2} \le B_\infty \le \gamma - \frac{1}{2}.\end{equation*}
On the other hand,
\begin{equation} \label{eq:rescaled_eq_fix}
    -v_L''(y) + W_L(y) v_L(y) = \hat{\mu}_L v_L(y), \quad y \in (A_L, B_L),
\end{equation}
where
\[
\hat{\mu}_L = L^2(\mu_1 - \pi^2),\qquad h_{L}(y):=h(Ly),\qquad W_L(y) = L^2 \pi^2 \left( \frac{1}{h_{L}^{2}(y)} - 1 \right).
\]
Equivalently, by rescaling the Rayleigh quotient for $\mu_1$, we see that $\hat{\mu}_L$ is the first eigenvalue of the operator $-\frac{d^2}{dy^2} + W_L$ on $(A_L, B_L)$, characterized by the variational principle:
\begin{equation*} \label{eq:rayleigh_rescaled}
    \hat{\mu}_L = \min_{\varphi \in H^1_0(A_L, B_L)} \frac{\ds\int_{A_L}^{B_L} |\varphi'(y)|^2 dy + \int_{A_L}^{B_L} W_L(y) \varphi^2(y) \, dy}{\ds\int_{A_L}^{B_L} \varphi^2(y) \, dy}.
\end{equation*}
The function $h_L(y)$ converges uniformly on any compact subinterval of $(A_{\infty},B_{\infty})$ to a limiting concave function $h_\infty(y)$ defined on $[A_\infty, B_\infty]$. 
Since $h_L(y) \ge 1 - 1/L^2$ on $[-1/2, 1/2]$, we uniformly obtain $h_\infty(y) \equiv 1$ on $[-1/2, 1/2]$.

Now, let us consider the eigenvalue equation \eqref{eq:rescaled_eq_fix} solved by $v_{L}$. Estimates \eqref{bound1dim} yield that $\hat{\mu}_L$ is uniformly bounded (in particular $\hat{\mu}_L \in (0, 3\pi^2+\eps)$ for $L$ large). Therefore, up to extracting a subsequence, $\hat{\mu}_L \to \hat{\mu}_\infty$. Being $W_L \ge 0$ and $v_L \le 1$, we deduce
\begin{equation*} \label{eq:weak_subsol}
    -v_L''(y) \le \hat{\mu}_L v_L(y) \le C, \quad \forall y \in (A_L, B_L).
\end{equation*}
Since $\phi_1$ is log-concave, $v_L$ is unimodal. By integrating from its maximum point and using the fact that the interval length $B_L - A_L \to \gamma$ is uniformly bounded, we deduce 
\[
|v_L'| \le C.
\] 
Thus, because $v_L$ continuously vanishes at $A_L$ and $B_L$, its zero-extension is uniformly Lipschitz continuous on $\mathbb{R}$. By the Ascoli-Arzelà theorem, $v_L$ converges, up to a subsequence and uniformly on $\mathbb{R}$, to a Lipschitz continuous function $v_\infty \in C^0(\mathbb{R})$.

Regarding the potential $W_L(y)$, notice that it is non-negative and convex since $h$ is positive and concave. Over the central interval $[-1/2, 1/2]$, we have $h_L(y) \ge 1 - 1/L^2$, which provides a uniform upper bound:
\begin{equation*}
    0 \le W_L(y) \le \pi^2 \frac{2 - 1/L^2}{(1 - 1/L^2)^2} \to 2\pi^2 \quad \forall y \in [-1/2, 1/2].
\end{equation*}
By standard properties of convex functions, the sequence $W_L$ converges, up to subsequences, pointwise to a convex function $W_\infty : (A_\infty, B_\infty) \to [0, +\infty]$. 
Define $J_{\textit{fin}} = \{ y \in (A_\infty, B_\infty) : W_\infty(y) < +\infty \}$. Due to convexity, $J_{\textit{fin}}$ is an interval. Since $W_\infty \le 2\pi^2$ on $[-1/2, 1/2]$, we have $[-1/2, 1/2] \subseteq J_{\textit{fin}}$. Let $(A^*, B^*)$ be the interior of $J_{\textit{fin}}$.

For any compact subset $K \subset (A^*, B^*)$, $W_L$ converges uniformly to $W_\infty$ and is uniformly bounded. 
Equation \eqref{eq:rescaled_eq_fix} then implies that $v_L''$ is uniformly bounded on $K$. 
Hence, $v_L$ converges in $C^1_{\text{loc}}(A^*, B^*)$ to $v_\infty$, and we can pass to the limit in \eqref{eq:rescaled_eq_fix} to obtain that $v_\infty \in C^1(A^*, B^*)$ is a solution to:
\begin{equation*}
    -v_\infty''(y) + W_\infty(y) v_\infty(y) = \hat{\mu}_\infty v_\infty(y) \quad \text{for } y \in (A^*, B^*).
\end{equation*}



Let us prove that $v_\infty(y) = 0$ for all $y \notin \overline{J_{\textit{fin}}}$. Multiplying the eigenvalue equation \eqref{eq:rescaled_eq_fix} by $v_L$ and integrating over $(A_L, B_L)$ yields:
\begin{equation} \label{eq:rayleigh}
    \int_{A_L}^{B_L} |v_L'(y)|^2 \, dy + \int_{A_L}^{B_L} W_L(y) |v_L(y)|^2 \, dy = \hat{\mu}_L \int_{A_L}^{B_L} |v_L(y)|^2 \, dy.
\end{equation}
Because $|v_L| \le 1$, the right-hand side is bounded by $\hat{\mu}_L (B_L - A_L)$. Since $\hat{\mu}_L$ is bounded and $B_L - A_L \to \gamma$, the left-hand side of \eqref{eq:rayleigh} is uniformly bounded. Hence by Fatou's Lemma,
\begin{equation*}
    \int_{A_\infty}^{B_\infty} W_\infty(y) |v_\infty(y)|^2 \, dy \le \liminf_{L \to \infty} \int_{A_L}^{B_L} W_L(y) |v_L(y)|^2 \, dy \le C.
\end{equation*}
 In order for the integral to be finite, we must necessarily have $v_\infty(y) = 0$ almost everywhere on $(A_\infty, B_\infty) \setminus J_{\textit{fin}}$. Being continuous, this immediately forces:
\begin{equation*}
    v_\infty(y) = 0 \quad \forall y \notin \overline{J_{\textit{fin}}}.
\end{equation*}

By continuity, this also enforces the Dirichlet boundary conditions $v_\infty(A^*) = v_\infty(B^*) = 0$ at the endpoints $A^{*},B^{*}$ of $J_{\textit{fin}}$.

It remains to bound the mass of $v_\infty$ inside $J_{\textit{fin}}$. Passing the rescaled eigenvalue equation to the limit, $v_\infty$ operates exactly as the first (positive, normalized) Dirichlet eigenfunction of the limit Schrödinger operator on the strict interior of $J_{\textit{fin}}$ (the fact that $v_\infty \ge 0$ does not change sign guarantees that $\hat{\mu}_\infty$ is indeed the principal eigenvalue):
\begin{equation*}
    \begin{cases}
    -v_\infty''(y) + W_\infty(y) v_\infty(y) = \hat{\mu}_\infty v_\infty(y) &\text{in } (A^*, B^*),\\[.3cm]
v_\infty(A^*) = v_\infty(B^*) = 0,\quad \max_{J_{\textit{fin}}} v_\infty = 1.
\end{cases}
\end{equation*}

Specifically, since at the maximum point $y_{L}$ of $v_L$ we have $v_L''(y_L) \le 0$, equation \eqref{eq:rescaled_eq_fix} gives $W_L(y_L) \le \hat{\mu}_L \le C$. This implies that any limit point of $y_L$ as $L \to \infty$ must belong to $\overline{J_{\textit{fin}}}$, hence the $L^\infty$ normalization is preserved in the limit and $\max_{J_{\textit{fin}}} v_\infty = 1$. 
Moreover, for any $y \in J_{\textit{fin}}$ we have $L^2(h_L(y)^{-2} - 1) = \pi^{-2}W_L(y) \to \pi^{-2} W_\infty(y) < \infty$, which definitively forces $h_L(y) \to 1$, meaning precisely that $h_\infty(y) = 1$ on $J_{\textit{fin}}$.

By Proposition \ref{prop:integral_bound}, its total mass satisfies:
\begin{equation*}
    \int_{A_\infty}^{B_\infty} v_\infty(y) \, dy = \int_{A^*}^{B^*} v_\infty(y) \, dy \le \frac{2}{\pi} (B^* - A^*). 
\end{equation*}
To conclude the proof, since $h_\infty(y) = 1$ on $J_{\textit{fin}}$,
\begin{equation*}
    \liminf_{L \to \infty} \frac{|\Omega_L|}{L} \ge \int_{A_\infty}^{B_\infty} h_\infty(y) \, dy \ge \int_{A^*}^{B^*} h_\infty(y) \, dy = B^* - A^*.
\end{equation*}
Finally, 
\begin{equation*}
    \limsup_{L \to \infty} \frac{1}{|\Omega_L|} \int_a^b \phi_1(x) \,dx 
    = \limsup_{L \to \infty} \frac{\ds\int_{A_L}^{B_L} v_L(y) \,dy}{\ds\int_{A_L}^{B_L} h_L(y) \,dy} 
    \le \frac{\ds\int_{A^*}^{B^*} v_\infty(y) \,dy}{\ds\int_{A^*}^{B^*} h_\infty(y) \,dy} \le\frac{2}{\pi}.
\end{equation*}
concluding the proof.
\end{proof}

\begin{remark}
In the analysis of case 1 of the above proof, we obtain classes of convex sets where the efficiency vanishes, at the limit; more precisely, these are the sets where the relevant length scale $L$ is $L=o(|\Omega|)$. Such result confirms, as already shown in \cite{vdBDPDBG}, that for elongating ellipses, or triangular domains, the efficiency vanishes.
\end{remark}

\subsection{Proof of Theorem \ref{thm:limeff}}
Relying on the Grieser-Jerison uncoupling theorem and Proposition \ref{prop:limit-uni}, we are now ready to prove our last two main results.
We retain the geometric normalization framework introduced at the beginning of Section 3.
\begin{proof}[Proof of Theorem \ref{thm:limeff}]
To simplify notation, we formally drop the sequence index and denote $\Omega_{n}$ simply as $\Omega$. Due to the scale invariance of $E(\Omega)$, there is no loss of generality in assuming $\Omega$ is normalized. Let $\sigma \in (0,1)$. We decompose the $L^1$ norm of $u$ based on its sublevel and superlevel sets:
\begin{equation}
\label{split}
    \int_\Omega u(x,y) \,dxdy = \int_{\Omega \cap \{u \le \sigma\}} u(x,y) \,dxdy + \int_{\Omega \cap \{u > \sigma\}} u(x,y) \,dxdy.
\end{equation}
For the first integral, we trivially have:
\begin{equation*} \label{eq:int_sigma}
    \int_{\Omega \cap \{u \le \sigma\}} u(x,y) \,dxdy \le \sigma |\Omega|.
\end{equation*}
For the second integral, let us consider $J_\sigma = \{x \in [a, b] : \max_y u(x,y) > \sigma\}$. Clearly, $\{u > \sigma\}$ is contained in $J_\sigma \times [0,1]$, thus
\begin{equation*} \label{eq:int_Jsigma}
    \int_{\Omega \cap \{u > \sigma\}} u(x,y) \,dxdy \le \int_{J_\sigma} \int_{f_1(x)}^{f_2(x)} u(x,y) \,dy \,dx.
\end{equation*}
As remarked in \cite[Section 2]{grieser_jerison}, by suitably adjusting the length fraction defining the central interval $I'$ and its structural constant, the Grieser-Jerison estimate can be naturally extended to any longitudinal segment where $\max_y u(x,y)$ admits a strictly positive lower bound. Thus, there exists a constant $C_\sigma > 0$ (depending only on $\sigma$) such that for all $x \in J_\sigma$:
\begin{equation*}
    \left| u(x,y) - \phi_1(x) \sin\left(\pi \frac{y-f_1(x)}{h(x)}\right) \right| \le \frac{C_\sigma}{L}.
\end{equation*}
Integrating we obtain:
\begin{align*}
    \int_{f_1(x)}^{f_2(x)} u(x,y) \,dy &\le \phi_1(x) \int_{f_1(x)}^{f_2(x)} \sin\left(\pi \frac{y-f_1(x)}{h(x)}\right) \,dy + \frac{C_\sigma}{L} h(x) \\
    &= \frac{2}{\pi} h(x) \phi_1(x) + \frac{C_\sigma}{L} h(x).
\end{align*}
Integrating this upper bound over the domain $J_\sigma$ yields:
\begin{equation*} \label{eq:int_Jsigma_bound}
    \int_{J_\sigma} \int_{f_1(x)}^{f_2(x)} u(x,y) \,dy \,dx \le \frac{2}{\pi} \int_{J_\sigma} \phi_1(x) \,dx + \frac{C_\sigma}{L} |J_\sigma|.
\end{equation*}
Given that $|J_\sigma| \le N \le 2|\Omega|$, we can insert these integral estimates back into \eqref{split} to find:
\begin{equation*}
    \|u\|_{L^1(\Omega)} \le \sigma |\Omega| + \frac{2}{\pi} \int_a^b \phi_1(x) \,dx + \frac{2C_\sigma}{L}|\Omega|.
\end{equation*}
Dividing by the area $|\Omega|$, the efficiency functional obeys:
\begin{equation*}
    E(\Omega) \le \sigma + \frac{2}{\pi |\Omega|} \int_a^b \phi_1(x) \,dx + \frac{2C_\sigma}{L}.
\end{equation*}
As $L \to \infty$, the trailing term naturally vanishes. Thanks to Proposition \ref{prop:limit-uni}, taking the superior limit yields:
\begin{equation*}
    \limsup_{L \to \infty} E(\Omega) \le \sigma + \frac{4}{\pi^2}.
\end{equation*}
Since $\sigma \in (0,1)$ was arbitrary, the thesis holds. 
\end{proof}

\subsection{Proof of Theorem \ref{thm:paynestakgoldmax}}

\begin{proof}[Proof of Theorem \ref{thm:paynestakgoldmax}]
An explicit computation on the unit open disk $B$ gives:
    \begin{equation*}
        E(B) = \frac{1}{|B|} \int_B J_0(j_{0,1} |x|) \,dx = 2 \int_0^1 r J_0(j_{0,1} r) \,dr = \frac{2}{j_{0,1}} J_1(j_{0,1}) \approx 0.4316.
   \end{equation*}
    Now, consider a maximizing sequence of domains $\{\Omega_n\} \subset \mathcal{C}$, where $\mathcal{C}$ denotes the broad class of all planar convex open sets. By definition:
        \begin{equation}
        \label{maxim}
        \lim_{n \to \infty} E(\Omega_n) = \sup_{\Omega \in \mathcal{C}} E(\Omega).
    \end{equation}
    We can assume the sequence elements $\Omega_n$ are normalized. We assert that the ratio $D_{n}/R_{n}$ must be uniformly bounded. Indeed, if this were false, up to extracting a subsequence we would have $D_{n} \to \infty$ (since $R_{n}$ is strictly bounded from above and below for normalized coordinates). In such an elongating scenario, Theorem \ref{thm:limeff} would dictate:
      \begin{equation*}
        \lim_{n \to \infty} E(\Omega_n) \le \frac{4}{\pi^2} \approx 0,4053 < E(B),
    \end{equation*}
which strictly contradicts the maximization limit \eqref{maxim}.

By virtue of the Blaschke Selection Theorem, the uniformly bounded sequence $\Omega_{n}$ converges (up to a subsequence) in the Hausdorff metric to a limit body $\Omega_*$. Since the domains' radii satisfy $c^{-1}\le R_{n} < D_{n} \le c$, the limit configuration $\Omega_*$ is non-degenerate and constitutes an admissible, bounded, planar convex open set. 
Consequently,
    \begin{equation*}
        E(\Omega_*) = \lim_{n \to \infty} E(\Omega_n) = \sup_{\Omega \in \mathcal{C}} E(\Omega).
    \end{equation*}
 The limit body $\Omega_*$ realizes the maximum for the efficiency functional, and the proof is completed.
\end{proof}
\section*{Acknowledgements}
This work has been partially supported by GNAMPA-INdAM.

\end{document}